\documentclass[a4paper]{article}
\usepackage{amsfonts}
\usepackage{amsmath}
\usepackage{amssymb}

\newcommand{\R}{ \mbox{\rm I}\!\mbox{\rm R}}

\newtheorem{theorem}{Theorem}[section]

\newtheorem{lemma}[theorem]{Lemma}

\newtheorem{remark}[theorem]{Remark}

\newenvironment{proof}[1][Proof]{\noindent\textbf{#1.} }{\ \rule{0.5em}{0.5em}}
\begin{document}

\title{ Approximate solution to abstract differential equations with variable domain}
\author{T.Ju.Bohonova \thanks{National Aviation University of Ukraine,
 1, Komarov ave. , 03058 Kyiv, Ukraine
({\tt bohonoff@astral.kiev.ua}).}, I.P. Gavrilyuk\thanks{ Staatliche
Studienakademie Th\"uringen, Berufsakademie Eisenach, University of
Cooperative Education,  Am Wartenberg 2, D-99817 Eisenach, Germany
({\tt ipg@ba-eisenach.de}).}, V.L.Makarov and V.Vasylyk
\thanks{Institute of Mathematics of NAS of Ukraine,
 3 Tereshchenkivs'ka Str., Kyiv-4, 01601, Ukraine
({\tt makarov@imath.kiev.ua}, {\tt vasylyk@imath.kiev.ua}).}
 }

\date{\null}
\maketitle

\begin{abstract}
A new exponentially convergent algorithm is proposed for an abstract the
first order differential equation with unbounded operator coefficient
possessing a variable domain. The algorithm is based on a generalization of
the Duhamel integral for vector-valued functions. This technique translates
the initial problem to a system of integral equations. Then the system is
approximated with exponential accuracy. The theoretical results are
illustrated by examples associated with the heat transfer boundary value
problems.
\end{abstract}



\section{Introduction}

This paper is devoted to special class of differential equations, which
are associated with the first order differential equation in a Banach space $X$

\begin{equation}  \label{in1}
\frac{du(t)}{dt}+A(t)u(t)=f(t), \; u(0)=u_{0}.
\end{equation}
Here $t$ is a real variable, the unknown function $u(t)$ and the given function $f(t)$ take values in $X$, and $A(t)$ is a given function whose values are densely defined, closed linear operators in $X$ with domains $D(A,t)$ depending on the parameter $t.$ Equations of the type \eqref{in1} are called abstract differential equations with an unbounded operator coefficient possessing a variable domain and can be considered as metamodels of initial boundary value problems for parabolic equations with time-depended boundary conditions.

The variable domain of an operator in some cases can be described by a separate equation, then we have an abstract problem of the kind 
\begin{equation}
\begin{split}
& \frac{du(t)}{dt}=A(t)u(t),\;0\leq s\leq t\leq T, \\
& L(t)u(t)=\Phi (t)u(t)+f(t),\;0\leq s\leq t\leq T, \\
& u(s)=u_{0}
\end{split}
\label{in6}
\end{equation}
instead of (\ref{in1}). Here $L(t)$ and $\Phi (t)$ are some linear operators defined on the boundary of the spatial domain and the second equation represent an abstract model of the time-dependent boundary condition. An existence and uniqueness result for this problem was proved in \cite{filali}.

The literature concerning discretization of such problems in abstract setting is rather not voluminous (see e.g. \cite{palencia3}, where the Euler difference approximation of the first accuracy order for problem (\ref{in1}) with the time-dependent domain was considered, and the references therein). It is clear that the discretization (with respect to $t$) is more complicated then in the case of a $t$-independent domain $D(A)$ since the inclusion $y_{k}=y(t_{k})\in D(A,t_{k})$ of the approximate solution $y_{k}$ at each discretization point $t_{k}$ should be additionally checked and guaranteed. The using of Duhamel integral was proposed in \cite{bgmv} for the problem

\begin{equation}\label{i1-1}
\begin{split}
& \frac{du(t)}{dt}+A(t)u(t)=f(t), \\
& \partial _{1}u(t)+\partial _{0}(t)u(t)=g(t),\; \\
& u(0)=u_{0},
\end{split}
\end{equation}
where $u(t)$ is the unknown function with values in a Banach space $X,$ $f(t)$ is a given measurable function, $A(t):D(A)\in X$ is a densely defined, closed linear operator in $X$ with a time-independent domain $D(A),$  $g(t)$ is a given function with values in some other Banach space $Y$ and $\partial _{1},$ $\partial _{0}(t)$ are linear operators. Here $u:(0,T)\rightarrow D(A)\subset X,$ $f:(0,T)\rightarrow X$ from $L_{q}(0,T;X)$ with the norm $\Vert f\Vert =\left\{ \int_{0}^{T}\Vert  f\Vert_{X}^{q}dt \right\} ^{1/q}$, $g:(0,T)\rightarrow Y$ is from $L_{q}(0,T;Y),$ and $\partial _{1}:D(A)\rightarrow Y$ (independent of $t$!), $\partial_{0} (t): D(A)\rightarrow Y$ (can depend on $t$!). Duhamel-like technique allows to transform the problem \eqref{i1-1} to a system of integral equations possessing operator coefficients with $t$-independent domains which can be efficiently approximate. It was constructed a discretization of high accuracy order for the case when $A(t)$ is a constant operator (i.e. $A(t)\equiv A$)  for the problem \eqref{i1-1} in \cite{bgmv} using this approach. 

The second equation above is just the time-dependent boundary condition with appropriate operators $\partial _{1},\partial _{0}$ acting on the boundary of the spatial domain. For this reason we call this equation an abstract (time-dependent) boundary condition. Including the boundary condition into the definition of the operator coefficient in the first equation we get a problem of the type (\ref{in1}) with a variable domain.

In the present paper we consider the problem (\ref{i1-1}) and build a new algorithm for approximate solution which rejects a limitation on the structure of the operator $\partial _{0}(t)$ used in \cite{bgmv} and can be suitable for non-constant operator $A(t).$

The paper is organized as follows. In Section \ref{sect2} we transform problem (\ref{i1-1}) to a system of abstract boundary integral equations using the Duhamel-like integral. In the Section \ref{Numalg} we construct a numerical method using the Tchebychev interpolation for involved unknown functions  and the collocation. The main theorem of this section shows an almost (i.e. within to a polynomial factor) exponential convergence of the discretization for analytical input data. In Section \ref{numex} we represent some computational experiment for our algorithm.

\section{Duhamel-like technique for the first order differential equations
in Banach space\label{sect2}}

For the problem (\ref{i1-1}) let us choose a mesh $\omega _{K}=\{t_{l}=l\ast \tau ,\;l=1,...,K,\;\tau = \frac{T}{K}\}$ of $K$ various points on $[0,T].$ Then one can rewrite the problem (\ref{i1-1}) in the following form:

\begin{equation}
\begin{split}
& \frac{du_{l}(t)}{dt}+A(t)u_{l}(t)=f(t),\quad t\in \left( t_{l-1},t_{l}\right], \\
& \partial _{1}u_{l}(t)+\partial _{0}(t)u_{l}(t)=g(t), \\
& u_{l}(t_{l-1})=u_{l-1}(t_{l-1}), \\
& l=1,2,\ldots ,K,
\end{split}
\label{rozbint}
\end{equation}
and 
\begin{equation*}
u(t)=u_{l}(t),\quad t\in \left[ t_{l-1},t_{l}\right] .
\end{equation*}

We transform each interval $[t_{l-1},t_{l}]$ to the $[-1,1]$ by change of
variables
\begin{equation*}
t=\frac{\tau }{2}s+\frac{t_{l}+t_{l-1}}{2}=\frac{\tau }{2}(s+2l-1)=\psi
_{l}(s),
\end{equation*}
then we obtain
\begin{equation}\label{rozbint2}
\begin{split}
& \frac{du_{l}(\psi_{l}(s))}{ds}+ \frac{\tau}{2} A(\psi_{l}(s)) u_{l}(\psi_{l}(s)) = \frac{\tau}{2} f(\psi_{l}(s)), \quad s\in \left( -1,1\right] , \\
& \partial _{1}u_{l}(\psi_{l}(s))+\partial _{0}(\psi_{l}(s))u_{l}(\psi_{l}(s))=g(\psi_{l}(s)), \\
& u_{l}(\psi_{l}(-1))=u_{l-1}(\psi_{l-1}(1)), \\
& l=1,2,\ldots ,K.
\end{split}
\end{equation}

Let us introduce the following notations:
\[
\begin{split}
v_{l}(s)= u_{l}(\psi_{l}(s)),\quad f_{l}(s)=f(\psi_{l}(s)),\\
\partial _{0,l}= \partial _{0}(\psi_{l}(s)), \quad g_{l}(s)= g(\psi_{l}(s)). 
\end{split}
\]

Further let us chose a mesh on segment $[-1,1]$ with the Chebyshev-Gauss-Lobatto nodes 
$\omega_{N}=\{s_{k}=\cos {\frac{(N-k)\pi }{N}},\;k=0,...,N\}.$ It is well known that $\max_k\{\theta_k\}=\frac{\pi}{N},$ $\theta_k=s_k-s_{k-1}.$ The problem (\ref{rozbint2}) is equivalent to 

\begin{equation} \label{fs2}
\begin{split}
& \frac{dv_{l}}{ds}+ \frac{\tau}{2}A_{l,k}v_{l}=\frac{\tau}{2} [A_{l,k}-A_l(s)]v_{l}+ \frac{\tau}{2}f_l(s), \\
&\partial_1v_l(s) = \partial_{0,l}v_l+g_l(s),\; s \in [-1,1],\\
&v_l(-1)=v_{l-1}(1),
\end{split}
\end{equation} 
where 
\begin{equation*}
A_{l,k}=A_l(s_k).
\end{equation*}

On each subinterval $(s_{k-1},s_k]$ we define the operator
$A^{(2)}_{l,k}$ with $t$-independent domain by
\begin{equation} \label{i47}
\begin{split}
& D(A^{(2)}_{l,k})=\{u \in D(A):\; \partial_1 u=0 \},\\
&A^{(2)}_{l,k} u=A_{l,k} u \; \forall u \in D(A^{(2)}_{l,k})
\end{split}
\end{equation}
and the operator  $B_{l,k}:Y \to D(A)$   by
\begin{equation}\label{i47-10}
\begin{split}
&A_{l,k}(B_{l,k} y)=0, \\
&\partial_1B_{l,k} y=y.
\end{split}
\end{equation}
For all $s\in [-1,1]$ we define the operators
\begin{equation} \label{i47-11}
\begin{split}
&A^{(2)}(s)=A^{(2)}_{l,k}, \; s \in (s_{k-1},s_k],\\
&B_{l}(s)=B_{l,k}, \; s \in (s_{k-1},s_k], \; \forall k=1,...,N.\\
\end{split}
\end{equation}

Further, we accept the following hypotheses:

{\bf (B1)} We suppose the operator $A^{(2)}(s)$ to be strongly
positive, i.e. there exists a positive constant $M_R$ independent
of $s$ such that on the rays and outside a sector $\Sigma_\theta
=\{z \in \mathbb C: 0 \le arg(z)\le \theta, \theta \in
(0,\pi/2)\}$ the following resolvent estimate holds
\begin{equation} \label{i390r}
\|(zI-A^{(2)}(s))^{-1}\|\le \frac{M_R}{1+|z|}.
\end{equation}
This assumption implies that there exists  positive constants $c,
\; \kappa$ such that \cite{fujita}, p.103
\begin{equation} \label{i400}
\|[A^{(2)}(s)]^\kappa e^{-\lambda A^{(2)}(s)}\|\le c \lambda^{-\kappa}, \lambda>0,
\kappa \ge 0 .
\end{equation}

{\bf (B2)} There exists a real positive $\omega$ such that
\begin{equation} \label{i410}
\|e^{-\lambda A^{(2)}(s)}\|\le e^{- \omega \lambda} \quad \forall \lambda ,s \in [-1,1]
\end{equation}
(see \cite{pazy}, Corollary 3.8, p.12, for corresponding
assumptions on $A(s)$ ).

We also assume that the following conditions hold:

{\bf (B3)}
\begin{equation} \label{i420}
\|[A^{(2)}(t)-A^{(2)}(s)][A^{(2)}(t)]^{-\gamma}\| \le c |t-s|
\quad \forall t,s, \; 0\le \gamma \le 1;
\end{equation}

{\bf (B4)}
\begin{equation} \label{i430}
\|[A^{(2)}(t)]^\beta[A^{(2)}(s)]^{-\beta}-I\|\le c|t-s| \quad
\forall t,s \in [-1,1].
\end{equation}

{\bf (B5)} 
\begin{equation} \label{fs1+1}
\|\partial_0 \|\le c.
\end{equation}

{\bf (B6)} It holds that
\begin{equation}\label{fs1+2}
\left[ \int_{-1}^t
\|[A^{(2)}(\eta)]^{1+\gamma}e^{-A^{(2)}(\eta)(t-\lambda)}B(\eta)\|^p_{Y
\to X}d \lambda \right]^{1/p} \le c \; \forall \; t, \; \eta \in
[-1,1], \; 0 \le \gamma.
\end{equation}

Following \cite{bgmv} one can wright down using the Duhamel's technique 

\begin{equation}\label{fs2-1}
\begin{split}
v_l(s)&=e^{-A^{(2)}_{l,k} \frac{\tau}{2}(s-s_{k-1})}v_l(s_{k-1})+ \frac{\tau}{2} \int_{s_{k-1}}^s
e^{-A^{(2)}_{l,k}\frac{\tau}{2}(s-\lambda)} \left\{-[A_l(\lambda)-A_{l,k}]v_l(\lambda)+ f_l(\lambda)\right\} d \lambda\\
&+ \frac{\tau}{2} \int_{s_{k-1}}^s
A^{(2)}_{l,k} e^{-A^{(2)}_{l,k}\frac{\tau}{2}(s-\lambda)}B_{l,k}\{-\partial_{0,l}(\lambda) v_l(\lambda) + g(\lambda)\} d \lambda , \\
\partial_{0,l}(s) & v_l(s)=\partial_{0,l}(s) e^{-A^{(2)}_{l,k} \frac{\tau}{2}(s-s_{k-1})}v_l(s_{k-1}) \\ 
& + \partial_{0,l}(s) \frac{\tau}{2} \int_{s_{k-1}}^s
e^{-A^{(2)}_{l,k}\frac{\tau}{2}(s-\lambda)} \left\{-[A_l(\lambda)- A_{l,k}]v_l(\lambda)+ f_l(\lambda)\right\} d \lambda\\
&+ \partial_{0,l}(s) \frac{\tau}{2} \int_{s_{k-1}}^s
A^{(2)}_{l,k} e^{-A^{(2)}_{l,k}\frac{\tau}{2}(s-\lambda)}B_{l,k}\{-\partial_{0,l}(\lambda) v_l(\lambda) + g(\lambda)\} d \lambda , \\
&s \in [s_{k-1},s_k], \; k=1,...,N,\\
&v_l(-1)=v_{l-1}(1).
\end{split}
\end{equation}

It was proved in \cite{bgmv} that under the assumptions {\bf B1}- {\bf B6} the system \eqref{fs2-1} possesses a unique solution in $\cal Y.$

\section{Numerical algorithm\label{Numalg}}

We use the interpolation on the Chebyshev-Gauss-Lobatto nodes in order to construct a discrete approximation of (\ref{fs2}), (\ref{fs2-1}). Let 
\begin{equation}
\begin{split}
& P_{N}(s;v_{l})=P_{N}v_{l}=\sum_{j=0}^{N}v_{l}(s_{j})L_{j,N}(s), \\
& P_{N}(s;\partial _{0,l}v_{l})=P_{N}(\partial
_{0,l}v_{l})=\sum_{j=0}^{N}\partial _{0,l}(s_{j})v_{l}(s_{j})L_{j,N}(s),
\end{split}
\label{s4}
\end{equation}
be the interpolation polynomials for $v_{l}(s),$ $\partial _{0,l}(s)v_{l}(s)$
on the mesh $\omega _{N},$ $x=(x_{0},x_{1},...,x_{N}),$ $x_{i}\in X,$\ and $y=(y_{0},y_{1},...,y_{N}),$ $y_{i}\in Y$ given vectors and 
\begin{equation}
P_{N}(s;y)=P_{N}y=\sum_{j=0}^{N}y_{j}L_{j,N}(s)  \label{s5}
\end{equation}%
the polynomial that interpolates $y$, where 
\begin{equation*}
L_{j,N}=\frac{T_{N}^{\prime }(s)(1-s^{2})}{\frac{d}{ds}[(1-s^{2})T_{N}^{\prime }(s)]_{s=s_{j}}(s-s_{j})},j=0,...,N
\end{equation*}
are the Lagrange fundamental polynomials. Substituting $P_{N}(\eta ;x)$ for $v_{l}(\eta ),$ $x_{k}$ for $v_{l}(s_{k}),$ $P_{N}(\eta ;y)$ for $\partial_{0,l}(\eta )v_{l}(\eta ),$ $y_{k}$ for $\partial _{0,l}(s_{k})v_{l}(s_{k})$ and then collocating in the points $s=s_{k}$ in (\ref{fs2-1}) we arrive at
the following system of linear equations with respect to the unknowns $x_{k},$ $y_{k}:$

\begin{equation}\label{fs6}
\begin{split}
& x_{k}^{(l)}=\mathrm{e}^{-A^{(2)}_{l,k}\frac{\tau }{2}\theta_k}x_{k-1}^{(l)}+ \sum_{j=0}^{N}\alpha _{kj}x_{j}^{(l)}+
\sum_{j=0}^{N}\beta_{kj}y_{j}^{(l)} + \phi _{k}^{(l)}, \\
& y_{k}^{(l)}=\partial _{0,l}(s_{k})\left[ \mathrm{e}^{-A^{(2)}_{l,k}\frac{\tau }{2}\theta_k}x_{k-1}^{(l)}+ \sum_{j=0}^{N}\alpha _{kj}x_{j}^{(l)}+
\sum_{j=0}^{N}\beta_{kj}y_{j}^{(l)}+ \phi _{k}^{(l)} \right], \\
& k=1,...,N;\; x_{0}^{(l)}=x_{N}^{(l-1)}=\tilde{v}_{l-1}(1),~y_{0}^{(l)}=y_{N}^{(l-1)}=\partial_{0,l-1}(1)\tilde{v}_{l-1}(1)
\end{split}
\end{equation}
which represents our algorithm. Here we use the notations

\begin{equation}\label{fs7}
\begin{split}
\alpha _{kj}& = \frac{\tau}{2} \int_{s_{k-1}}^{s_k}
e^{-A^{(2)}_{l,k}\frac{\tau}{2}(s_k - \eta)} \left\{A_{l,k}-A_l(\eta)\right\} L_{j,N}(\eta)  d \eta ,\\
 \beta_{kj}& =-\frac{\tau }{2}\int_{s_{k-1}}^{s_{k}}A^{(2)}_{l,k}\mathrm{e}^{-A^{(2)}_{l,k}
\frac{\tau }{2}(s_{k}-\eta )}B_{l,k} L_{j,N}(\eta )d\eta , \\
\phi _{k}& =\frac{\tau }{2}\left( \int_{s_{k-1}}^{s_{k}}A^{(2)}_{l,k}\mathrm{e}
^{-A^{(2)}_{l,k}\frac{\tau }{2}(s_{k}-\eta )}B_{l,k} g_{l}(\eta )d\eta + \int_{s_{k-1}}^{s_{k}}
\mathrm{e}^{-A^{(2)}_{l,k}\frac{\tau }{2}(s_{k}-\eta )}f_{l}(\eta )d\eta \right) ,
\end{split}
\end{equation}
and suppose that we have an algorithm to compute these coefficients.

\begin{remark}
\label{ObchKoef}Under the assumption that $f(t),g(t)$ are polynomials the
calculation of the ope\-ra\-tors $\alpha _{kj}$ and the elements $\phi _{k}$
can be reduced to the calculation of integrals of the kind $%
I_{s}=\int_{t_{k-1}}^{t_{k}}\mathrm{e}^{-A_{k}^{(2)}(t_{k}-\lambda )}\lambda
^{s}d\lambda ,$ which can be found by a simple recurrence algorithm: $%
I_{l}=-l\left[ A_{k}^{(2)}\right] ^{-1}I_{l-1}+\left[ A_{k}^{(2)}\right]
^{-1}\left( t_{k}^{l}I-t_{k-1}^{l}e^{-A_{k}^{(2)}\tau _{k}}\right)
,l=1,2,...,s,$ $I_{0}=\left[ A_{k}^{(2)}\right] ^{-1}\left(
I-e^{-A_{k}^{(2)}\tau _{k}}\right) ,$ where the operator exponentials can be
computed by the exponentially convergent algorithm from \cite{gm5}, \cite{vas2}.\bigskip
\end{remark}

\bigskip After separating of $x_{0}^{(l)}$ and $y_{0}^{(l)}$ in (\ref{fs6}) (that we
assume are known from the previous step) we have

\begin{equation}\label{fs6-s}
\begin{split}
& x_{k}^{(l)}=\mathrm{e}^{-A^{(2)}_{l,k}\frac{\tau }{2}\theta_k}x_{k-1}^{(l)}+
\alpha _{k0}x_0^{(l)} +\beta_{k0} y_{0}^{(l)} + \sum_{j=1}^{N}\alpha _{kj}x_{j}^{(l)} +\sum_{j=1}^{N}\beta_{kj}y_{j}^{(l)}+\phi_{k}^{(l)}, \\
& y_{k}^{(l)}=\partial _{0,l}(s_{k})\left[ \mathrm{e}^{-A^{(2)}_{l,k}\frac{\tau }{2}\theta_k}x_{k-1}^{(l)}+
\alpha _{k0}x_0^{(l)} +\beta_{k0} y_{0}^{(l)} + \sum_{j=1}^{N}\alpha _{kj}x_{j}^{(l)} +\sum_{j=1}^{N}\beta_{kj}y_{j}^{(l)}+\phi_{k}^{(l)}\right] , \\
& k=1,...,N;\;x_{0}^{(l)}=x_{N}^{(l-1)},~y_{0}^{(l)}=y_{N}^{(l-1)},
\end{split}
\end{equation}

For errors $z_{x}^{(l)}=(z_{x,1}^{(l)},...,z_{x,N}^{(l)}),$ $z_{y}^{(l)}=(z_{y,1}^{(l)},...,z_{y,N}^{(l)})$ with $z_{x,k}^{(l)}=v_{l}(s_{k})-x_{k}$ and $z_{y,k}^{(l)}=\partial_{0,l}(s_{k})v_{l}(s_{k})-y_{k}$ we have relations

\begin{equation}
\begin{split}
z_{x,k}^{(l)}& = \alpha _{k0}z_{x,N}^{(l-1)}+ \beta_{k0}z_{y,N}^{(l-1)} +\mathrm{e}^{-A^{(2)}_{l,k} \frac{\tau }{2} \theta_{k}} z_{x,k-1}^{(l)}+ \sum_{j=1}^{N}\alpha _{kj}z_{x,j}^{(l)}  +\sum_{j=1}^{N}\beta_{kj}z_{y,j}^{(l)} +\psi _{k}^{(l)}, \\
z_{y,k}^{(l)}& =\partial _{0,l}(s_{k})\left[ \alpha _{k0}z_{x,N}^{(l-1)}+ \beta_{k0}z_{y,N}^{(l-1)} +\mathrm{e}^{-A^{(2)}_{l,k} \frac{\tau }{2} \theta_{k}} z_{x,k-1}^{(l)}+ \sum_{j=1}^{N}\alpha _{kj}z_{x,j}^{(l)}  +\sum_{j=1}^{N}\beta_{kj}z_{y,j}^{(l)} +\psi _{k}^{(l)} \right] \\
k& =1,...,N;
\end{split}
\label{poh}
\end{equation}
where 
\begin{equation}\label{fs9}
\begin{split}
\psi _{k}^{(l)}= \frac{\tau}{2} \int_{s_{k-1}}^{s_k}
e^{-A^{(2)}_{l,k}\frac{\tau}{2}(s_k - \eta)} \left\{A_{l,k}-A_l(\eta)\right\} \left\{ v_l(\eta) -P_{N}(\eta;v_{l}) \right\}  d \eta  \\
 -\frac{\tau }{2}\int_{s_{k-1}}^{s_k}A^{(2)}_{l,k}\mathrm{e}^{-A^{(2)}_{l,k}\frac{\tau }{2}(s_{k} -\eta)}B_{l,k} [\partial _{0,l}(\eta )v_{l}(\eta ) -P_{N}(\eta;\partial _{0,l} v_{l})]d\eta . 
\end{split} 
\end{equation}

In order to represent algorithm (\ref{fs6-s}) in a block-matrix form
we introduce the following matrix and vectors:
\begin{equation*}
S^{(l)}=\{s_{i,k}\}_{i,k=1}^N=
  \begin{pmatrix}
    E_X & 0 & 0 & \cdot & \cdot & \cdot & 0 & 0 \\
    -e^{-A_{l,2}^{(2)} \frac{\tau}{2}\theta_2} & E_X & 0 & \cdot & \cdot & \cdot & 0 & 0 \\
    0 & -e^{-A_{l,3}^{(2)} \frac{\tau}{2}\theta_3} & E_X & \cdot & \cdot & \cdot & 0 & 0 \\
    \cdot & \cdot & \cdot & \cdot & \cdot & \cdot & \cdot & \cdot \\
    0 & 0 & 0 & \cdot & \cdot & \cdot & -e^{-A_{l,N}^{(2)} \frac{\tau}{2}\theta_{N}} & E_X \
  \end{pmatrix},
\end{equation*}
\begin{equation*}
F_x^{(l)}=
  \begin{pmatrix}
    [A_{l,1}^{(2)}]^{\gamma}e^{-A_{l,1}^{(2)} \frac{\tau}{2}\theta_1}+ [A_{l,1}^{(2)}]^{\gamma}\alpha_{1,0} \\
    [A_{l,2}^{(2)}]^{\gamma}\alpha_{2,0} \\
    \vdots \\
    [A_{l,N}^{(2)}]^{\gamma}\alpha_{N,0}  \
  \end{pmatrix},
\end{equation*}
\begin{equation*}
F_y^{(l)}=
  \begin{pmatrix}
    [A_{l,1}^{(2)}]^{\gamma}\beta_{1,0} \\
    [A_{l,2}^{(2)}]^{\gamma}\beta_{2,0} \\
    \vdots \\
    [A_{l,N}^{(2)}]^{\gamma}\beta_{N,0} \
  \end{pmatrix},
\end{equation*}
with $E_X$ being the identity operator in $X,$ the matrix
$C^{(l)}=\{\tilde{\alpha}_{k,j}\}_{k,j=1}^N$ with
$\tilde{\alpha}_{k,j}=[A_{l,k}^{(2)}]^\gamma
\alpha_{k,j}[A_{l,j}^{(2)}]^{-\gamma}$, the matrix
$D^{(l)}=\{\tilde{\beta}_{k,j}\}_{k,j=1}^N$ with
$\tilde{\beta}_{k,j}=[A_{l,k}^{(2)}]^\gamma \beta_{k,j}$ and the
vectors
\begin{equation}\label{fs11}
\begin{split}
 & \tilde{x}^{(l)}=\begin{pmatrix}
    [A_{l,1}^{(2)}]^\gamma x_1^{(l)} \\
    \cdot \\
    \cdot \\
    \cdot \\
    [A_{l,N}^{(2)}]^\gamma x_N^{(l)} \
  \end{pmatrix}, \quad
  f_x^{(l)}=\begin{pmatrix}
    [A_{l,1}^{(2)}]^\gamma \phi_1^{(l)} \\
    \cdot \\
    \cdot \\
   \cdot \\
    [A_{l,N}^{(2)}]^\gamma \phi_N^{(l)} \
  \end{pmatrix}, \\
 & \tilde{\psi}^{(l)}=\begin{pmatrix}
    [A_{l,1}^{(2)}]^\gamma \psi_{1}^{(l)} \\
    \cdot \\
    \cdot \\
    \cdot \\
    [A_{l,N}^{(2)}]^\gamma \psi_{N}^{(l)} \
  \end{pmatrix}.
  \end{split}
\end{equation}
Besides, we introduce the matrix
\begin{equation}\label{fs11-1}
\tilde{S}^{(l)}=\{s_{i,k}\}_{i,k=1}^N=
  \begin{pmatrix}
    E_X & 0 & 0 & \cdot & \cdot & \cdot & 0 & 0 \\
    - \tilde{s}_{1}& E_X & 0 & \cdot & \cdot & \cdot & 0 & 0 \\
    0 & -\tilde{s}_{2} & E_X & \cdot & \cdot & \cdot & 0 & 0 \\
    \cdot & \cdot & \cdot & \cdot & \cdot & \cdot & \cdot & \cdot \\
    0 & 0 & 0 & \cdot & \cdot & \cdot & -\tilde{s}_{N-1}& E_X \
  \end{pmatrix},
\end{equation}
with $\tilde{s}_{i-1}=e^{-A_{l,i}^{(2)}\frac{\tau}{2}\theta_i}[A_{l,i}^{(2)}]^{\gamma}[A_{l,i-1}^{(2)}]^{-\gamma},
i=2,...,N.$

It is easy to see that for the (left) inverse
\begin{equation}\label{fs13}
\begin{split}
&(\tilde{S}^{(l)})^{-1}=\{\tilde{s}_{i,k}^{(-1)}\}_{i,k=1}^N\\
&=\begin{pmatrix}
 E_X & 0&   \cdots & 0 & 0 \\
 \tilde{s}_{1} &E_X  &  \cdots & 0 & 0 \\
 \tilde{s}_{2}\tilde{s}_{1} &\tilde{s}_{2}  &  \cdots & 0 & 0 \\
  \cdot & \cdot &  \cdots & \cdot & \cdot \\
  \tilde{s}_{N-1} \cdots \tilde{s}_{1} & \tilde{s}_{N-1}\cdots \tilde{s}_{2}
   & \cdots & \tilde{s}_{N-1} & E_X
\end{pmatrix}
\end{split}
\end{equation}
it holds
\begin{equation}\label{fs12}
  (\tilde{S}^{(l)})^{-1}\tilde{S}^{(l)}=\begin{pmatrix}
    E_X & 0 & \cdots & 0 \\
    0 & E_X & \cdots & 0 \\
    \cdot &\cdot & \cdots & \cdot \\
    0& 0 & \cdots & E_X \
  \end{pmatrix}.
\end{equation}
\begin{remark}
Using results of \cite{ gm5, ghk,ghk3} one can get a parallel and
sparse approximations with an exponential convergence rate of the
operator exponentials contained in $(\tilde{S}^{(l)})^{-1}$ and as a
consequence a parallel and sparse approximation of
$\tilde{S}^{-1}.$
\end{remark}

We multiply the first equation in  (\ref{fs6-s}) and the first equation in (\ref{poh}) by $[A_{l,k} ^{(2)}]^\gamma$ to obtain a solution of \eqref{fs6-s} and estimating of error. Then, \eqref{fs6-s}, \eqref{poh} can be written in the matrix form as follows:

\begin{equation}
\begin{split}
\tilde{S}^{(l)}\tilde{x}^{(l)}& = C^{(l)} \tilde{x}^{(l)} + D^{(l)}y^{(l)} +F_x^{(l)}{x}^{(l)}_0 +F_y^{(l)}y^{(l)}_0 + f_x^{(l)}, \\
y^{(l)}& =\Lambda \left[ (I-\tilde{S}^{(l)})\tilde{x}^{(l)} +C^{(l)} \tilde{x}^{(l)} + D^{(l)}\tilde{y}^{(l)} +F_x^{(l)}{x}^{(l)}_0 +F_y^{(l)}y^{(l)}_0 + f_x^{(l)} \right] ,
\end{split}
\label{fs10}
\end{equation}

\begin{equation}
\begin{split}
\tilde{S}^{(l)}\tilde{z}_{x}^{(l)}& = C^{(l)} \tilde{z}_{x}^{(l)} + D^{(l)}z_{y}^{(l)} +F_x^{(l)}{z}_{x,N}^{(l-1)} +F_y^{(l)}z_{y,N}^{(l-1)} + \tilde{\psi}^{(l)}, \\
z_{y}^{(l)}& =\Lambda \left[ (I-\tilde{S}^{(l)})\tilde{z}_{x}^{(l)} +C^{(l)} \tilde{z}_{x}^{(l)} + D^{(l)}z_{y}^{(l)} +F_x^{(l)}{z}_{x,N}^{(l-1)} +F_y^{(l)}z_{y,N}^{(l-1)}+ \tilde{\psi}^{(l)} \right] ,
\end{split}
\label{pohm}
\end{equation}
where

\begin{equation*}
\Lambda =diag\left[ \partial _{0,l}(s_{1})[A_{l,1}^{(2)}]^{-\gamma},...,\partial
_{0,l}(s_{N})[A_{l,N}^{(2)}]^{-\gamma}\right] .
\end{equation*}

The systems \eqref{fs10} and \eqref{pohm} are equivalent to the following ones:

\begin{equation}
\begin{split}
\tilde{S}^{(l)}\tilde{x}^{(l)}& = C^{(l)} \tilde{x}^{(l)} + D^{(l)}y^{(l)} +F_x^{(l)}\tilde{x}^{(l-1)}_N +F_y^{(l)}y^{(l-1)}_N + f_x^{(l)}, \\
y^{(l)}& =\Lambda \left[ (I-\tilde{S}^{(l)})\tilde{x}^{(l)} +C^{(l)} \tilde{x}^{(l)} + D^{(l)}\tilde{y}^{(l)} +F_x^{(l)}\tilde{x}^{(l-1)}_N +F_y^{(l)}y^{(l-1)}_N + f_x^{(l)} \right] ,
\end{split}
\label{fs10-m}
\end{equation}

For a vector $v=(v_{1},v_{2},...,v_{N})^{T}$ and a block operator matrix $A=\{a_{ij}\}_{i,j=1}^{N}$ we introduce the vector norm 
\begin{equation}
|\Vert v\Vert |\equiv |\Vert v\Vert |_{\infty }=\max_{1\leq k\leq N}\Vert
v_{k}\Vert  \label{fs21}
\end{equation}
and the consistent matrix norm 
\begin{equation}
|\Vert A\Vert |\equiv |\Vert A\Vert |_{\infty }=\max_{1\leq i\leq
N}\sum_{j=1}^{N}\Vert a_{i,j}\Vert .  \label{fs22}
\end{equation}

For further analysis we need the following auxiliary result.

\begin{lemma}\label{l1}
Under assumptions {\bf B1}- {\bf B6}  the following estimates hold
true
\begin{equation}\label{l100}
\begin{split}
&|\|(\tilde{S}^{(l)})^{-1}\|| \le  c N,\\
&|\|C^{(l)}\||\le   c (\frac{\tau}{2})^{2-\gamma}  N^{\gamma-2} \ln{N} ,\\
&|\|D^{(l)}\||\le   c\frac{\tau}{2} N^{-1/q}\ln{N}, \; 1/p+1/q=1,\\
& |\| \Lambda \||\le c.
\end{split}
\end{equation}
with a positive constant $c$ independent of $N.$
\end{lemma}
\begin{proof}
{Due to {\bf
(B4)} we have
$$|\|[A_{l,k}^{(2)}]^{\gamma}[A_{l,k-1}^{(2)}]^{-\gamma}\| |= |
\|[A_{l,k}^{(2)}]^{\gamma}[A_{l,k-1}^{(2)}]^{-\gamma}-E_X+E_X|\|\le
1+c \frac{\tau}{2}\theta_k.$$ Using this estimate, $\max_k \theta_k\le\frac{\pi}{N}$ and {\bf (B2)} } we get further if $\frac{\tau}{2}\le 1$
\begin{equation}\label{fs23}
\begin{split}
&|\|(\tilde{S}^{(l)})^{-1}\|| \le 1+e^{-\omega \frac{\tau}{2}\frac{1}{N}}(1+c\frac{\tau}{2}\frac{1}{N}) +\cdots+[e^{-\omega \frac{\tau}{2}\frac{1}{N}}(1+c \frac{\tau}{2}\frac{1}{N})]^{N-1}\\
&\le 1+(1+c \frac{1}{N})+\cdots+(1+c \frac{1}{N})^{N-1} \le \frac{e^{2 c}}{c \frac{1}{N}} \le c N.
\end{split}
\end{equation}

Using {\bf B2}, {\bf B3} for $C^{(l)}$ we have
\[
|\|C^{(l)}\||\le \max_{1 \le k \le N}\sum_{j=1}^N
\|\tilde{\alpha}_{kj}\|
\]
\[
=\max_{1 \le k \le N}\sum_{j=1}^N \frac{\tau}{2}
\|\int_{s_{k-1}}^{s_k}[A_{l,k}^{(2)}]^{\gamma}e^{-A_{l,k}^{(2)}\frac{\tau}{2}(s_k-\eta)} [A_{l,k}^{(2)} - A_l(\eta)] L_{j,N}(\eta) [A_{l,j}^{(2)}]^{-\gamma}  d \eta\|
\]
\[
\le \frac{\tau}{2} \max_{1 \le k \le N}\sum_{j=1}^N 
\int_{s_{k-1}}^{s_k}\|[A_{l,k}^{(2)}]^{\gamma}e^{-A_{l,k}^{(2)}\frac{\tau}{2}(s_k-\eta)} \| \, \| [A_{l,k}^{(2)} - A_l(\eta)] [A_{l,k}^{(2)}]^{-\gamma}\| \, \|[A_{l,k}^{(2)}]^{\gamma}  [A_{l,j}^{(2)}]^{-\gamma} \| \, | L_{j,N}(\eta)| d \eta
\]
\[
\le \frac{\tau}{2} \,  \max_{1 \le k \le N} 
\int_{s_{k-1}}^{s_k} (\frac{\tau}{2}(s_k-\eta))^{-\gamma} \frac{\tau}{2}(s_k-\eta) c \sum_{j=1}^N| L_{j,N}(\eta)| d \eta
\]
\[
\le c \Lambda_N (\frac{\tau}{2})^{2-\gamma}  \max_{1 \le k \le N} 
\int_{s_{k-1}}^{s_k} (s_k-\eta)^{1-\gamma} d \eta
\]
\[
\le c (\frac{\tau}{2})^{2-\gamma}  N^{\gamma-2} \ln{N},
\]
where \[\Lambda_n=\max_{-1 \le \tau \le
1}\sum_{j=1}^n|L_{j,N}(\tau)|\] 
is the Lebesgue constant related
to the Chebyshev-Gauss-Lobatto interpolation nodes.
  For the matrix $D^{(l)}$ we have from
(\ref{fs1+2})
\[
|\|D^{(l)}\||\le \max_{1 \le k \le N}\sum_{j=1}^N
\|\tilde{\beta}_{kj}\| 
\]
\[=\max_{1 \le k \le N}\sum_{j=1}^N \frac{\tau}{2}
\|\int_{s_{k-1}}^{s_k}[A_{l,k}^{(2)}]^{1+\gamma}e^{-A_{l,k}^{(2)}\frac{\tau}{2}(s_k-\eta)} B_{l,k}^{(2)}
L_{j,N}(\eta) d \eta\| 
\]
\[
\le \frac{\tau}{2} \, \max_{1 \le k \le N}\int_{s_{k-1}}^{s_k}\|[A_{l,k}^{(2)}]^{1+\gamma}e^{-A_{l,k}^{(2)}\frac{\tau}{2}(s_k-\eta)} B_{l,k}^{(2)} \| \sum_{j=1}^N |L_{j,N}(\eta) |d \eta 
\]
\[
\le c \frac{\tau}{2} \Lambda_N
 \int_{s_{k-1}}^{s_k} \|[A_{l,k}^{(2)}]^{1+\gamma} e^{-A_{l,k}^{(2)}\frac{\tau}{2}(s_k-\eta)} B_{l,k}^{(2)} \|
d \eta \le c\frac{\tau}{2} N^{-1/q}\ln{N}.
\]

The last estimate is a simple consequence of assumptions {\bf
(B1)} and {\bf (B5)}. The lemma is proved.
\end{proof}

From the second equation in (\ref{fs10-m}) one can write down

\begin{equation*}
[ I-\Lambda D^{(l)} ] y^{(l)}= \Lambda \left[ I-\tilde{S}^{(l)} +C^{(l)} \right] \tilde{x}^{(l)} +\Lambda \Phi^{(l)},
\end{equation*}
where 
\[
\Phi^{(l)}= F_x^{(l)}{x}^{(l-1)}_N +F_y^{(l)}y^{(l-1)}_N + f_x^{(l)}.
\]
If exists $[I-\Lambda D^{(l)}]^{-1}$ we have 
\begin{equation*}
 y^{(l)}= [ I-\Lambda D^{(l)} ]^{-1} \Lambda \left[ I-\tilde{S}^{(l)} +C^{(l)} \right] \tilde{x}^{(l)} + [ I-\Lambda D^{(l)} ]^{-1} \Lambda \Phi^{(l)}.
\end{equation*}
Otherwise one can choose appropriate $\tau $ so that $|\| \Lambda D^{(l)} \| |<1$ and in this case it means that there exists operator-matrix $[I-\Lambda D^{(l)}]^{-1}.$ Substituting this expression into the first equation in (\ref{fs10-m}) we have 
\[
G^{(l)} \tilde{x}^{(l)} = Q^{(l)} \Phi^{(l)} , 
\]
where
\[\begin{split}
&G^{(l)}=\tilde{S}^{(l)}- C^{(l)} - D^{(l)} [ I-\Lambda D^{(l)} ]^{-1} \Lambda [ I-\tilde{S}^{(l)} +C^{(l)} ], \\
&Q^{(l)}= D^{(l)} [ I-\Lambda D^{(l)} ]^{-1} \Lambda + I_X .
\end{split}
\]

Similarly one can obtain from \eqref{pohm} 

\begin{equation*}
\begin{split}
 &z_y^{(l)}= [ I-\Lambda D^{(l)} ]^{-1} \Lambda \left[ I-\tilde{S}^{(l)} +C^{(l)} \right] \tilde{z}_x^{(l)} + [ I-\Lambda D^{(l)} ]^{-1} \Lambda \tilde{\Psi}^{(l)},\\
 &G^{(l)} \tilde{z}_x^{(l)} = Q^{(l)} \tilde{\Psi}^{(l)},
\end{split}
\end{equation*}
where
\[
\tilde{\Psi}^{(l)}= F_x^{(l)}z_{x,N}^{(l-1)} +F_y^{(l)}z_{y,N}^{(l-1)} + \tilde{\psi}^{(l)}
\]

\begin{lemma}\label{l2}
Under assumptions of Lemma \ref{l1} there exists $(G^{(l)})^{-1}$ and it
holds
\begin{equation} \label{fs29}
\begin{split}
&|\|(G^{(l)})^{-1}\|| \le c N , \\
&|\|Q^{(l)} \|| \le c
\end{split}
\end{equation}
with some constant independent on $N$.
\end{lemma}

\begin{proof}
We represent $G^{(l)}=\tilde{S}^{(l)}[I_X-G_1^{(l)}]$ and estimate $|\|G_1^{(l)}|\|$ with 
\[
G_1^{(l)}=(\tilde{S}^{(l)})^{-1}C^{(l)}+(\tilde{S}^{(l)})^{-1}D^{(l)}[I_Y-\Lambda
D^{(l)}]^{-1} \Lambda (I_X - \tilde{S}^{(l)} + C^{(l)}).
\]

  We have in the case when exists $[I-\Lambda D^{(l)}]^{-1}$ (this can be always achieved, see comments above)
\[
|\|G_1^{(l)}\|| \le |\|(\tilde{S}^{(l)})^{-1}\||\cdot |\|C^{(l)}\||+|\|(\tilde{S}^{(l)})^{-1}\|| \cdot
 |\|D^{(l)}\|| c \||\Lambda\||
\cdot (\||I_X-\tilde{S}^{(l)}\||+\||C^{(l)}\||)
\]
and now Lemma \ref{l1} implies
\begin{equation} \label{fs30}
|\|G_1^{(l)}\|| \le c \ln{N} \left(
\frac{1}{N^{1-\gamma}} \left(\frac{\tau}{2}\right)^{2-\gamma} +\frac{1}{N^{1/q-1}}\frac{\tau}{2} \right).
\end{equation}
This estimate guarantees the existence of the bounded inverse
operator $(I_X -G_1)^{-1}$ by the appropriate choose of $\tau$ (to provide $|\|G_1^{(l)}\|| <1$) which together with the estimate
$|\|(\tilde{S}^{(l)})^{-1}\|| \le c N$ proves the first assertion of the
lemma. The second assertion is evident. The proof is complete.
\end{proof}

This lemma and representations of $\tilde{x}^{(l)},$ $y^{(l)},$ $\tilde{z}^{(l)}_x$ and $z_y^{(l)}$ imply the following stability estimates:
\begin{equation} \label{fs31}
\begin{split}
&\|| \tilde{x}^{(l)}\||\le c N \||\Phi^{(l)}\||,\\
& \||\tilde{z_x}^{(l)}\||\le c N\||\tilde{\psi}^{(l)}\||.
\end{split}
\end{equation}
\begin{equation} \label{fs32}
\begin{split}
&\||y^{(l)}\|| \le c N \||\Phi^{(l)}\||,\\
&\||z_y^{(l)}\|| \le c N\||\tilde{\psi}^{(l)}\||.
\end{split}
\end{equation}

Let $\Pi_{N}$ be the set of all polynomials in $t$ with vector
coefficients of degree less or equal then $N .$ In complete
analogy with \cite{babenko, szegoe, szegoe1} the following
Lebesgue inequality for vector-valued functions can be proved
\begin{equation}\label{fs32n}
\| u(\eta)-P_{N}(\eta; u)\|_{C[-1,1]}\equiv \max_{\eta \in
[-1,1]}\| u(\eta)-P_{N}(\eta; u)\| \le (1+\Lambda_N)E_N(u)
\end{equation}
with the error of the best approximation of $u$ by polynomials of
degree not greater then $N$
\begin{equation}\label{fs33}
E_N( u)=\underset{p \in \Pi_{N}}{\text{inf}}\max_{\eta \in
[-1,1]}\| u(\eta)-p(\eta)\|.
\end{equation}

Now, we can estimate the error of our algorithm for $l'$s stage.

\begin{theorem}\label{main0-0}
Let the assumptions of Lemma \ref{l1} with $\gamma<1$ hold, then
there exists a positive constant $c$ such that
\begin{enumerate}
\item For $N,$ $K$ large enough it holds
\begin{equation}\label{fs27n}
\begin{split}
&|\|\tilde{z}_x^{(l)}\||\le c \left\{\left(\frac{\tau}{2}\right)^{2-\gamma}N^{1-\gamma}\ln{N} E_N([A_{l,0}]^{\gamma}\tilde{v}_l) + \frac{\tau }{2}N^{1-1/q} \ln{N} E_N(\partial \tilde{v}_l)\right\},\\
&|\|z_y^{(l)}\||\le c \left\{\left(\frac{\tau}{2}\right)^{2-\gamma}N^{1-\gamma}\ln{N} E_N([A_{l,0}]^{\gamma} \tilde{v}_l) + \frac{\tau }{2}N^{1-1/q} \ln{N} E_N(\partial \tilde{v}_l)\right\}
\end{split}
\end{equation}
where $\tilde{v}_l$ is the solution of \eqref{fs2-1} with the initial condition $\tilde{v}_{l-1}(1)$;
  \item The equation for $\tilde{x}^{(l)}$  can be written in the form
  \begin{equation} \label{fs36}
  \tilde{x}^{(l)}=G_1^{(l)}\tilde{x}^{(l)} +[\tilde{S}^{(l)}]^{-1} Q^{(l)} \Phi^{(l)}
\end{equation}
and can be solved by the fixed point iteration
\begin{equation} \label{fs37}
\tilde{x}^{(l)}_{(k+1)}=G_1^{(l)}\tilde{x}^{(l)}_{(k)} +[\tilde{S}^{(l)}]^{-1} Q^{(l)} \Phi^{(l)}, \quad
k=0,1,...; \tilde{x}^{(l)}_{(0)}- \text{arbitrary}
\end{equation}
with the convergence rate of an geometrical progression with the
denominator $$q \le c \ln{N} \left(
\frac{1}{N^{1-\gamma}} \left(\frac{\tau}{2}\right)^{2-\gamma} +\frac{1}{N^{1/q-1}}\frac{\tau}{2} \right)<1$$ for $N,$ $K$ large enough.
\end{enumerate}
\end{theorem}
\begin{proof}
For $\tilde{z}_x^{(l)}$ we have to estimate $\tilde{\psi}_x^{(l)}$ in (\ref{fs31}).

\begin{equation*}
\begin{split}
|\|\tilde{\psi}_{x}^{(l)} \||=&\max_{1 \le k \le N} \left\| \frac{\tau}{2} \int_{s_{k-1}}^{s_k} [A^{(2)}_{l,k}]^{\gamma}e^{-A^{(2)}_{l,k}\frac{\tau}{2}(s_k - \eta)} \left\{A_{l,k}-A_l(\eta)\right\} \left\{ \tilde{v}_l(\eta) -P_{N}(\eta; \tilde{v}_{l}) \right\}  d \eta \right. \\
 & \left. -\frac{\tau }{2} \int_{s_{k-1}}^{s_k} [A^{(2)}_{l,k}]^{\gamma+1}\mathrm{e}^{-A^{(2)}_{l,k} \frac{\tau }{2}(s_{k} -\eta)}B_{l,k} [\partial _{0,l}(\eta )\tilde{v}_{l}(\eta ) -P_{N}(\eta;\partial _{0,l}\tilde{v}_{l})]d\eta \right\| 
\end{split}
\end{equation*}
\begin{equation*}
\begin{split}
\le &\max_{1 \le k \le N} \left\| \frac{\tau}{2} \int_{s_{k-1}}^{s_k} [A^{(2)}_{l,k}]^{\gamma} e^{-A^{(2)}_{l,k}\frac{\tau}{2}(s_k - \eta)} \left\{A_{l,k}-A_l(\eta)\right\}[A_{l,k}]^{-\gamma} [A_{l,k}]^{\gamma} [A_{l,0}]^{-\gamma}\right. \\ 
&\left. \qquad \qquad \times \left\{ [A_{l,0}]^{\gamma}\tilde{v}_l(\eta) -P_{N}(\eta;[A_{l,0}]^{\gamma} \tilde{v}_{l}) \right\}  d \eta \right\| \\
 &+\max_{1 \le k \le N}\left\| \frac{\tau }{2}\int_{s_{k-1}}^{s_k} [A^{(2)}_{l,k}]^{\gamma+1}\mathrm{e}^{-A^{(2)}_{l,k}\frac{\tau }{2}(s_{k} -\eta)}B_{l,k} [\partial _{0,l}(\eta )\tilde{v}_{l}(\eta ) -P_{N}(\eta;\partial _{0,l} \tilde{v}_{l})]d\eta \right\| 
\end{split}
\end{equation*}
\begin{equation*}
\begin{split}
\le &c(\frac{\tau}{2})^{2-\gamma} \max_{1 \le k \le N} 
 \int_{s_{k-1}}^{s_k} (s_k - \eta)^{1-\gamma }d \eta \, \left\| [A_{l,0}]^{\gamma}\tilde{v}_l(\cdot) -P_{N}(\cdot;[A_{l,0}]^{\gamma} \tilde{v}_{l}) \right\|_{C[-1,1]}  \\
 &+\max_{1 \le k \le N} \frac{\tau }{2}\theta_k^{1/q} \left\| \partial \tilde{v}_l(\cdot) -P_{N}(\cdot;\partial  \tilde{v}_{l}) \right\|_{C[-1,1]} ,  
\end{split}
\end{equation*}
where $\partial \tilde{v}_{l}(\eta)= \partial _{0,l}(\eta ) \tilde{v}_{l}(\eta ). $
Further, using \eqref{fs32n} we obtain
\begin{equation}\label{ocpsi}
|\|\tilde{\psi}_{x}^{(l)} \|| \le  c(\frac{\tau}{2})^{2-\gamma}N^{2-\gamma}\ln{N} E_N([A_{l,0}]^{\gamma} \tilde{v}_l) + c\frac{\tau }{2}N^{-1/q} \ln{N} E_N(\partial \tilde{v}_l). 
\end{equation}

The first assertion of the theorem follows from (\ref{fs31})
and the second one from (\ref{fs32}). \end{proof}

Let us estimate the full error of approximation in the collocating points.
\[
|\|z^{(l)}\| | = \max_{1\le k \le N} \|v_l(s_k) -x_k^{(l)}\| \le \max_{1\le k \le N} \|v_l(s_k) - \tilde{v}_l(s_k)\| + \max_{1\le k \le N} \|\tilde{v}_l(s_k) -x_k^{(l)}\|,
\]
where the first summand in the right-hand side of inequality is the error cosed by approximation of the initial data for the $l'$s step of our algorithm. Let $z_v^{(l)}=(v_l(s_1) - \tilde{v}_l(s_1)),\dots ,v_l(s_N) - \tilde{v}_l(s_N).$ Therefore
\[
\begin{split}
|\|z^{(l)}\| | &\le |\|z_v^{(l)}\| | + | \|z_x^{(l)}\|| \le {\rm e}^{c\tau}|\|z_v^{(l-1)}\| | + | \|z_x^{(l)}\|| ={\rm e}^{c\tau}|\|z^{(l-1)}\| | + | \|z_x^{(l)}\|| \\
&\le \cdots \le \sum_{j=1}^{l}{\rm e}^{(l-j)c\tau} | \|z_x^{(j)}\|| \le \sum_{j=1}^{K}{\rm e}^{(K-j)c\tau} | \|z_x^{(j)}\||
\end{split}
\]
The same is valid for the error $|\|z_{\partial}^{(l)}\| |= \max_{1\le k \le N} \|\partial_{0,l}((s_k)v_l(s_k) -y_k^{(l)}\|$
\[
|\|z_{\partial}^{(l)}\| | \le \sum_{j=1}^{K}{\rm e}^{(K-j)c\tau} | \|z_y^{(j)}\||.
\]
Let us introduce the following notation $z=(z^{(1)},\dots ,z^{(K)}),$ $z_\partial =(z_\partial^{(1)},\dots ,z_\partial^{(K)}).$ Now we can formulate the main result. 
\begin{theorem}\label{main0}
Let the assumptions of theorem \ref{main0-0} hold, then there exists a positive constant $c$ such that for $N,$ $K$ large enough it holds
\begin{equation}\label{fs27nnnn}
\begin{split}
&|\|\tilde{z}_x^{(l)}\||\le c \left\{\left(\frac{\tau}{2}\right)^{2-\gamma}N^{1-\gamma}\ln{N} E_N([A_{l,0}]^{\gamma}\tilde{v}_l) + \frac{\tau }{2}N^{1-1/q} \ln{N} E_N(\partial \tilde{v}_l)\right\},\\
&|\|z_y^{(l)}\||\le c \left\{\left(\frac{\tau}{2}\right)^{2-\gamma}N^{1-\gamma}\ln{N} E_N([A_{l,0}]^{\gamma} \tilde{v}_l) + \frac{\tau }{2}N^{1-1/q} \ln{N} E_N(\partial \tilde{v}_l)\right\}
\end{split}
\end{equation}
where $\tilde{v}_l$ is the solution of \eqref{fs2-1} with the initial condition $\tilde{v}_{l-1}(1).$
\end{theorem}
\begin{proof}
From the estimates for $|\|z^{(l)}\| |$ and $|\|z_{\partial}^{(l)}\| |$ we have
\begin{equation}
\begin{split}
|\|z \| | & \le \sum_{j=1}^{K}{\rm e}^{K c\tau} | \|z_x^{(j)}\|| \le {\rm e}^{c}\sum_{j=1}^{K} c \left\{\left(\frac{\tau}{2}\right)^{2-\gamma}N^{1-\gamma}\ln{N} E_N([A_{j,0}]^{\gamma}\tilde{v}_j) + \frac{\tau }{2}N^{1-1/q} \ln{N} E_N(\partial \tilde{v}_j)\right\} \\
&\le c \sum_{j=1}^{K} \frac{\tau}{2} \left\{\left(\frac{\tau}{2}\right)^{1-\gamma}N^{1-\gamma}\ln{N} E_N([A_{j,0}]^{\gamma}\tilde{v}_j) + N^{1-1/q} \ln{N} E_N(\partial \tilde{v}_j)\right\}.
\end{split}
\end{equation}
The same is valid for $z_\partial$
\begin{equation}
\begin{split}
|\|z_\partial \| | \le c \sum_{j=1}^{K} \frac{\tau}{2} \left\{\left(\frac{\tau}{2}\right)^{1-\gamma}N^{1-\gamma}\ln{N} E_N([A_{j,0}]^{\gamma}\tilde{v}_j) + N^{1-1/q} \ln{N} E_N(\partial \tilde{v}_j)\right\}.
\end{split}
\end{equation}

\end{proof}

\section{Numerical example}\label{numex}

In this section we show that the algorithm (\ref{fs6}) possesses the
exponential convergence  with respect to the temporal
discretization parameter $n$ predicted by Theorem \ref{main0}. In
order to eliminate the influence of other errors (the spatial error,
the error of approximation of the operator exponential and of the
integrals in (\ref{fs7})) we calculate the coefficients of the
algorithm (\ref{fs6}) exactly using the computer algebra tool
Maple.

A special example of the problem from the class (\ref{in1}) is
\begin{equation}\label{1}
\begin{split}
& \frac{\partial u}{\partial t}=\frac{\partial^2 u}{\partial
x^2}+f(x,t), \\
&u(0,t)=0, \frac{\partial u(1,t)}{\partial x}+b(t)u(1,t)=g(t),\\
&u(x,0)=u_0(x),
\end{split}
\end{equation}
where the operator $A:D(A)\in X \to X, \; X=L_q(0,1)$ is defined
by
\begin{equation}\label{2}
\begin{split}
& D(A)=\{v \in W_q^2(0,1): v(0)=0\}, \\
&Av=-\frac{\partial^2 v}{\partial x^2},
\end{split}
\end{equation}
the operators  $\partial_1:D(A) \to Y,$ and $\partial_0(t):D(A) \to
Y,$ $Y=\R $   are defined by
\begin{equation}\label{3}
\begin{split}
&\partial_1u=\frac{\partial u(x,t)}{\partial x}\left |_{x=1} \right.,\\
&\partial_0(t)u=b(t)\cdot u(x,t)|_{x=1}
\end{split}
\end{equation}
and $g(t) \in L_q(0,T;Y)= L_q(0,T)$.

As it was shown in \cite{bgmv} the second boundary integral equation in \eqref{fs2-1} for the example problem (\ref{1})
takes the form
\begin{equation}\label{17}
\begin{split}
&b(t) u(1,t)=b(t) v(1,t)+b(t) \int_0^t \frac{\partial}{\partial
t}W_1(1,\lambda,t-\lambda)d \lambda \\
&=b(t) v(1,t)-b(t) \int_0^t K(t-\lambda)g(\lambda)d
\lambda+b(t)\int_0^t K(t-\lambda)b(\lambda)u(1,\lambda)d \lambda,
\end{split}
\end{equation}
with
\begin{equation}\label{18}
\begin{split}
K(t)&=2 \sum_{n=1}^\infty \mathrm{e}^{-[\pi (2 n-1)/2]^2t}.
\end{split}
\end{equation}
The first equation in \eqref{fs2-1} is presented as follows

\begin{equation}\label{17f}
\begin{split}
&u(x,t)= v(x,t) - \int_0^t K_1(t-\lambda,x)g(\lambda)d
\lambda + \int_0^t K_1(t-\lambda,x)b(\lambda)u(1,\lambda)d \lambda,
\end{split}
\end{equation}
with
\begin{equation}\label{18f}
\begin{split}
K_1(t,x)=2 \sum_{n=1}^\infty (-1)^{n+2}\mathrm{e}^{-[\pi (2 n-1)/2]^2t} \sin (\frac{\pi}{2}(2n-1)x),\\
v(x,t)= 2 \sum_{n=1}^\infty \mathrm{e}^{-[\pi (2 n-1)/2]^2t} \sin (\frac{\pi}{2}(2n-1)x) \int_0^1 u_0(\xi) \sin (\frac{\pi}{2}(2n-1)\xi)d\xi .
\end{split}
\end{equation}

\begin{remark} Note that in this particular case we can
represent the integrand through the kernel $K_1(t-\lambda,x)$ ( obviously that $K(t-\lambda)=K_1(t-\lambda,1)$)
analytically. In general case one can use the exponentially convergent algorithm for the operator exponential in \eqref{fs7} like the ones from \cite{gm5}, \cite{vas2}.
\end{remark}

Let us consider the particular case of the problem \eqref{1}, when 
\[
b(t)= \mathrm{e}^{-\frac{\pi^2}{2}t}, \quad g(t)=\mathrm{e}^{-3\frac{\pi^2}{4}t}, \quad u_0(x)=\sin(\frac{\pi}{2}x),
\]
with exact solution
\[
u(x,t)=\mathrm{e}^{-\frac{\pi^2}{4}t}\sin(\frac{\pi}{2}x).
\]

Further we use to the equations \eqref{17}, \eqref{17f} the  collocation method
\eqref{fs6-s}.

The results of computations are presented in the tables (\ref{tab:M2})-(\ref{tab:M8}). The first column indicates the collocation point $t_k,$ in the second column there are the errors of approximation in the boundary $z_{x,k}$ (i.e. for $u(1,t)$), in the third column there are the errors of approximation in the collocation points for $x=\frac{1}{2}.$

\begin{table*}[h]
    \begin{center}
        \begin{tabular}{|c|c||c|}
      \hline
      Point $t$ & $\varepsilon_1$ & $\varepsilon_2$ \\
      \hline
            .8535533905& .28334234e-2 &  .19383315e-2 \\
            \hline
            .1464466094& .139098462e-1 & .75016794e-2  \\
            \hline
        \end{tabular}
    \end{center}

    \caption{The error in the case $n=2,$  $T=1$}
    \label{tab:M2}
\end{table*}

\begin{table*}[h]
    \begin{center}
        \begin{tabular}{|c|c||c|}
      \hline
      Point $t$ & $\varepsilon_1$ & $\varepsilon_2$ \\
      \hline
            0.9619397662 & .36662211e-4 & .23005566e-4 \\
            \hline
            0.6913417161 & .34443339e-4 & .33521073e-4\\
            \hline
            0.3086582838 & .42572982e-3 & .25374395e-3 \\
            \hline
            0.0380602337 & .23840042e-3 & .22946416e-4 \\
            \hline
        \end{tabular}
    \end{center}

    \caption{The error in the case $n=4,$ $T=1$}
    \label{tab:M4}
\end{table*}

\begin{table*}[h]
    \begin{center}
        \begin{tabular}{|c|c||c|}
      \hline
      Point $t$ & $\varepsilon_1$ & $\varepsilon_2$ \\
      \hline
            0.9903926402 & .46943895e-9 & .14327614e-9 \\
            \hline
            0.9157348061 & .14308953e-9 & .22763917e-9\\
            \hline
             0.7777851165 & .1564038e-8 & .45939028e-9\\
            \hline
             0.5975451610 & .2439358e-8 & .13481294e-9\\
            \hline
             0.4024548389 & .9823893e-8 & .28128559e-8\\
            \hline
             0.2222148834 & .1794445e-7 & .27310794e-8\\
            \hline
             0.0842651938 & .3218373e-7 & .61284010e-8\\
            \hline
             0.0096073597 & .1009601e-7 & .30125449e-11\\
            \hline
        \end{tabular}
    \end{center}

    \caption{The error in the case $n=8,$ $T=1$}
    \label{tab:M8}
\end{table*}

{\bf Acknowledgment}. The authors would like to acknowledge the support provided by the German Research Foundation (Deutsche Forschungsgemeinschaft -- DFG).

\bibliographystyle{plain}

\end{document}